\pdfoutput=1
\documentclass{article}
\usepackage[utf8]{inputenc}
\usepackage{amsmath}
\usepackage{amssymb}
\usepackage{amsthm}
\usepackage{tikz-cd}
\usepackage{mathtools}
\usepackage{hyperref}
\usepackage[totalwidth=480pt, totalheight=680pt]{geometry}
\usepackage{enumitem}

\newtheorem{thm}{Theorem}[section]
\newtheorem{corollary}[thm]{Corollary}

\newtheorem{lemma}[thm]{Lemma}
\theoremstyle{definition}

\newtheorem{example}[thm]{Example}
\newtheorem{examples}[thm]{Examples}
\newtheorem{remark}[thm]{Remark}
\newtheorem*{ack}{Acknowledgements}
\newtheorem*{conventions}{Conventions}

\setlist[enumerate]{label=(\arabic*)}

\title{On automorphisms of $p$-torsion $\mathbf{G}_m$-gerbes}
\author{Noah Olander}
\date{}

\begin{document}
\maketitle

\begin{abstract}
    Rouquier proved that for a smooth projective variety $X$, the group scheme $\operatorname{Pic}^0_X \rtimes \operatorname{Aut}^0_X$ is an invariant of the derived category of $X$. This was generalized to the twisted case by Olsson, who associated a group algebraic space $\operatorname{Aut}^0_{\mathcal{X}}$ to a $\mathbf{G}_m$-gerbe $\mathcal{X} \to X$ and proved it to be a twisted derived invariant. In characteristic zero, Olsson showed that $\operatorname{Aut}^0_{\mathcal{X}}$ is an extension of $\operatorname{Aut}^0_X$ by a subgroup scheme of $\operatorname{Pic}_X$. It is important in applications to compute $\operatorname{Aut}^0_{\mathcal{X}}$ in positive characteristic as well, which is the problem we consider here. We use deformation theory and representability results of Bragg and Olsson to prove that in many cases of interest, Olsson's description of $\operatorname{Aut}^0_{\mathcal{X}}$ as an extension also holds in characteristic $p$. As a corollary, we show that twisted derived equivalent abelian varieties are isogenous, simultaneously generalizing results of Honigs and Lane. We also provide an example showing that Olsson's extension description does not hold in general.
    This uses a result of Illusie on crystalline and flat cohomology and a result of Yang on the Brauer group of an ordinary variety, which we generalize to cohomology in arbitrary degree.
\end{abstract}

\section{Introduction}

Let $k$ be an algebraically closed field and $X, Y$ smooth projective varieties over $k$. Rouquier proved in \cite{MR2806466} that given a Fourier--Mukai equivalence
$$
\Phi_K : D^b_{Coh}(X) \xrightarrow{\cong} D^b_{Coh}(Y)
$$
between the bounded derived categories of $X$ and $Y$, there is an associated isomorphism of $k$-group schemes
$$
\operatorname{Pic}^0_X \rtimes \operatorname{Aut}^0_X \xrightarrow{\cong} \operatorname{Pic}^0_Y \rtimes \operatorname{Aut}^0_Y,
$$
where in the semi-direct products, automorphisms act on the right on line bundles by pullback. This result has proven foundational in the study of derived categories of varieties. For example, it is used in the paper \cite{Honigs} to show that derived equivalent abelian varieties are isogenous.\footnote{Strictly speaking, this only uses the special case of Rouquier's result when $X, Y$ are abelian varieties, which was shown earlier by Orlov in \cite[Theorem 2.19]{orlov2002abelian}.} See also \cite{PopaSchnell}, \cite{caucci2019derived}, \cite{caucci2023derived} for papers which heavily use Rouquier's work.

Rouquier's result was recently generalized by Olsson in \cite{MR4961251} to the case of twisted derived categories. 
Namely, if $\mathcal{X}, \mathcal{Y}$ are $\mathbf{G}_m$-gerbes over $X, Y$ with cohomology classes $\alpha \in H^2(X, \mathbf{G}_m), \beta \in H^2(Y, \mathbf{G}_m)$, respectively, then given any Fourier--Mukai equivalence 
$$
\Phi_K : D^b_{Coh}(X, \alpha) \xrightarrow{\cong} D^b_{Coh}(Y, \beta)
$$
between the $\alpha$- and $\beta$-twisted bounded derived categories of $X$ and $Y$, there is an associated isomorphism of $k$-group algebraic spaces
$$
\operatorname{Aut}^0_{\mathcal{X}} \to \operatorname{Aut}^0_{\mathcal{Y}},
$$
where $\operatorname{Aut}^0_{\mathcal{X}}$ is the connected component of the identity of the $k$-group algebraic space $\operatorname{Aut}_{\mathcal{X}}$ parametrizing (classes up to natural isomorphism of) automorphisms of $\mathcal{X}$ over $X$ inducing the identity on stabilizer groups; and similarly for $\operatorname{Aut}^0_{\mathcal{Y}}$. See Section \ref{section-prelims} for details. 

Olsson's result indeed generalizes Rouquier's because when $\alpha = 0$, so $\mathcal{X} \to X$ is a trivial gerbe, 
$$
\operatorname{Aut}_{\mathcal{X}} \cong \operatorname{Pic}_X \rtimes \operatorname{Aut}_X.
$$
See \cite[Example 1.2]{MR4961251}. For general $\mathcal{X}$, Olsson proves that there is an exact sequence
\begin{equation}
\label{equn-theexactsequence}
1 \to \operatorname{Pic}_X \to \operatorname{Aut}_{\mathcal{X}} \to \operatorname{Aut}_X
\end{equation}
of $k$-group algebraic spaces, and that if the order of $\alpha \in H^2(X, \mathbf{G}_m)$ is invertible in $k$, then the associated homomorphism of identity components
\begin{equation}
\label{equn-thehomomorphism}
\operatorname{Aut}^0_{\mathcal{X}} \to \operatorname{Aut}^0_X
\end{equation}
is surjective. This 
surjectivity allows one to compute $\operatorname{Aut}^0_{\mathcal{X}}$ in practice and has therefore been crucial in applications of Olsson's work, for example in \cite{lane2026twistedderivedequivalencesabelian}, where among other things, the author uses Olsson's results to show that twisted derived equivalent abelian varieties over the complex numbers are isogenous. This partially generalizes  the result of Honigs discussed above. More applications can be found in the papers \cite{li2026pointobjectsderivedequivalences} and \cite{perry2026semiregularitytheoremequivariantnoncommutative}. In all these cases, the authors use surjectivity of $\operatorname{Aut}^0_{\mathcal{X}} \to \operatorname{Aut}^0_X$, and so they are forced to assume that the order of the Brauer class is invertible in the base field (or work in characteristic zero).

In this paper, we study the morphism $\operatorname{Aut}^0_{\mathcal{X}} \to \operatorname{Aut}^0_X$ when the order of $\alpha$ is divisible by the characteristic $p>0$ of $k$. We first prove in Section \ref{section-sufficient} that in many cases of interest, surjectivity continues to hold.

\begin{thm}[See Theorem \ref{thm-sufficientconditions}]
    Assume any of the following conditions hold:
\begin{enumerate}
    \item  $H^0(X, T_X) = 0$,
    \item $H^2(X, \mathcal{O}_X) = 0$, 
    \item $\operatorname{Aut}^0_X$ is an abelian variety (for example, if $X$ is an abelian variety).
\end{enumerate}
Then
$
\operatorname{Aut}^0_{\mathcal{X}} \to \operatorname{Aut}^0_X
$
is surjective.
\end{thm}

The proof uses results from \cite{bragg2025representabilitycohomologyfiniteflat} on the representability of higher direct images of flat group schemes. As an application, we generalize the aforementioned result of Lane to arbitrary characteristic.

\begin{corollary}[See Corollary \ref{corollary-applicationtoabelianvarieties}]
   Suppose $X, Y$ are abelian varieties, $\alpha \in \operatorname{Br}(X), \beta \in \operatorname{Br}(Y)$, and there is a Fourier--Mukai equivalence $D^b_{\operatorname{Coh}}(X, \alpha) \cong D^b_{\operatorname{Coh}}(Y, \beta)$. Then $X$ and $Y$ are isogenous. 
\end{corollary}

We additionally construct an example showing that $
\operatorname{Aut}^0_{\mathcal{X}} \to \operatorname{Aut}^0_X
$
is not surjective in general.

\begin{thm}[See Theorem \ref{thm-mainpaperversion}]
\label{thm-main}
The homomorphism (\ref{equn-thehomomorphism}) of $k$-group algebraic spaces need not be surjective when the order of $\alpha$ is equal to the characteristic of $k$. In fact, if $k$ has characteristic $2$, then there exist a smooth projective variety $X$ and $\mathbf{G}_m$-gerbe $\mathcal{X} \to X$ with the following properties:
\begin{enumerate}
    \item $X$ has dimension 4, is ordinary in the sense of Kato, and has trivial tangent bundle.
    \item The class $\alpha \in H^2(X, \mathbf{G}_m)$ of $\mathcal{X}$ has order $2$. 
    \item The homomorphism of $k$-group algebraic spaces $\operatorname{Aut}^0_{\mathcal{X}} \to \operatorname{Aut}^0_X$ is not surjective. 
\end{enumerate}
\end{thm}

In fact, this occupies the bulk of the paper (Sections \ref{section-completetorsion}--\ref{section-proofmaintheorem}), and while the construction of the example is not difficult, the proof that it works uses a number of technical results. Using deformation theory, we show that it suffices to find an $X$ for which there exist an $\alpha \in H^2(X, \mathbf{G}_m)$ and a $\theta \in H^0(X, T_X)$ such that $\theta(\operatorname{dlog}(\alpha))$ is not in the subspace $H^1(\operatorname{Spec}k[\epsilon], \operatorname{Pic}_X) \subset H^2(X, \mathcal{O}_X)$. See Section \ref{section-prelims} for details. Thus, we seek an $X$ for which the $k$-span of the image of $$
\operatorname{dlog} : H^2(X, \mathbf{G}_m) \to H^2(X, \Omega_X)
$$
is large, as is the vector space $H^0(X, T_X)$, while the subspace $H^1(\operatorname{Spec}k[\epsilon], \operatorname{Pic}_X) \subset H^2(X, \mathcal{O}_X)$ is small. We take for $X$ a quotient of a product $A \times B$ of ordinary abelian varieties over $k$ by the free action of $\mathbf{Z}/2$ given by multiplication by $(-1)$ on the first factor and by translation by a $2$-torsion point on the second. This is a natural generalization of Igusa's famous surface from \cite{MR74085}, but the constraints at hand require us to pick $d = \operatorname{dim}(A)$ and $e = \operatorname{dim}(B)$ so that
$$2de-d > d(e+d)$$
(and thus not $d = 1 = e$ as in \cite{MR74085}, see Remark \ref{remark-classicaligusasurface}). To show that the $k$-span of the image of $\operatorname{dlog} : H^2(X, \mathbf{G}_m) \to H^2(X, \Omega_X)$ is large, we use results from \cite{MR565469} on the relation between flat and crystalline cohomology and a result from \cite{MR4936530} stating that the $p$-power torsion subgroup of the Brauer group of an ordinary smooth proper variety in characteristic $p$ is isomorphic to a direct sum of a finite group and $(\mathbf{Q}_p/\mathbf{Z}_p)^{\oplus e}$ for some integer $e \geq 0$. See Section \ref{section-flatcohomologyordinary}. We are able to prove the following slight generalization of Yang's result using Greenlees--May duality between derived complete and torsion objects of $D(\mathbf{Z})$ (recalled in Section \ref{section-completetorsion}).

\begin{thm}[See Theorem \ref{lemma-cofinite}]
Assume $X$ is ordinary. Then for all $i \geq 1$ there are an integer $e \geq 0$ and a finite Abelian $p$-group $F$ such that the $p$-power torsion subgroup of $H^i(X, \mathbf{G}_m)$ is isomorphic to 
    $
    (\mathbf{Q}_p/\mathbf{Z}_p)^{\oplus e} \oplus F.
    $
\end{thm}

\begin{conventions}
    We generally follow the conventions on algebraic spaces and stacks of \cite{stacks-project}. If $S$ is a scheme, $X$ is a scheme (or algebraic space) over $S$, and $G \to S$ is a flat group scheme locally of finite presentation, then  $H^i(X, G)$ denotes the cohomology of the sheaf represented by $G$ on the site $(\operatorname{Sch}/X)_{fppf}$. If $G$ is smooth over $S$, this is the same as the cohomology of the sheaf represented by $G$ on the small \'etale site of $X$ by a theorem of Grothendieck \cite[Th\'eor\`eme 11.7]{MR244271}. By $\mathbf{G}_m$-gerbe we mean a gerbe \emph{banded by $\mathbf{G}_m$}. This is the data of a gerbe $\mathcal{X}$ over an algebraic space $X$ in the sense of \cite[\href{https://stacks.math.columbia.edu/tag/06QB}{Tag 06QB}]{stacks-project} \emph{together with} for every object $x \in \mathcal{X}$ lying over a scheme $T$, an isomorphism of groups $H^0(T, \mathcal{O}_T)^\times \to \operatorname{Aut}_{\mathcal{X}_T}(x)$ such that these isomorphisms are compatible with morphisms $x \to y$ in $\mathcal{X}$. An \emph{equivalence of $\mathbf{G}_m$-gerbes over an algebraic space $X$} means an equivalence compatible with these isomorphisms. With these conventions, $\mathbf{G}_m$-gerbes over $X$ are classified up to equivalence by the group $H^2(X, \mathbf{G}_m)$, see Giraud \cite[IV]{MR344253}. Also, for an abelian group (or sheaf) $A$ and integer $n$, we denote by $A[n]$ the kernel of the homomorphism $A \to A$ given by multiplication by $n$. For an abelian variety $X$ we write $X^t$ for the dual abelian variety.
\end{conventions}

\begin{ack}
    I am grateful to Martin Olsson for many enlightening conversations and to him and Danny Bragg for sharing a draft of a forthcoming paper from which I learned about the fppf sheafified variants of the Artin--Mazur formal groups. This work was partially supported by the National Science Foundation under grant 
DMS-2402087.
\end{ack}

\section{Preliminaries}
\label{section-prelims}

Let $k$ be an algebraically closed field and $X$ a smooth proper variety over $k$. Let $\mathcal{X} \to X$ be a $\mathbf{G}_m$-gerbe (please see our conventions on $\mathbf{G}_m$-gerbes). Consider the category fibered in groupoids $\mathcal{A}ut_{\mathcal{X}}$ over $\operatorname{Sch}/k$  whose fiber over a $k$-scheme $T$ is the groupoid of autoequivalences $\mathcal{X}_T \to \mathcal{X}_T$ of stacks over $T$ which induce the identity on bands. Olsson shows in \cite{MR4961251} that this is a $\mathbf{G}_m$-gerbe over a $k$-group algebraic space $\operatorname{Aut}_{\mathcal{X}}$, which necessarily represents the fppf sheafification of the functor taking a $k$-scheme $T$ to the set of isomorphism classes of automorphisms $\mathcal{X}_T \to \mathcal{X}_T$ of algebraic stacks over $T$ which induce the identity on bands. Since an automorphism $\mathcal{X}_T \to \mathcal{X}_T$ inducing the identity on bands induces a morphism of associated $\mathbf{G}_m$-rigidifications, there is a homomorphism of $k$-group algebraic spaces $\operatorname{Aut}_{\mathcal{X}} \to \operatorname{Aut}_X$. Olsson proves in \emph{loc. cit.} that this homomorphism fits into a left exact sequence
\begin{equation}
\label{equn-olssonsequence}
0 \to \operatorname{Pic}_X \to \operatorname{Aut}_{\mathcal{X}} \to \operatorname{Aut}_X
\end{equation}
of $k$-group algebraic spaces.

Write $\pi : (Sch/X)_{fppf} \to (Sch/k)_{fppf}$ for the structure morphism. Then there is a right action 
$$
R^2\pi_{ *}\mathbf{G}_m \times \operatorname{Aut}_X \to R^2\pi_{ *}\mathbf{G}_m, \hspace{3 em} (x, \varphi) \mapsto \varphi^*x
$$
which makes $R^2\pi_{ *}\mathbf{G}_m$ into a right $\operatorname{Aut}_X$-module. Let $\alpha \in H^2(X, \mathbf{G}_m) = R^2\pi_{ *}\mathbf{G}_m(k)$ be the cohomology class of $\mathcal{X}$. Then there is a morphism of sheaves of sets
$$
a : \operatorname{Aut}_X \to R^2\pi_{ *}\mathbf{G}_m, \hspace{3 em} \varphi \mapsto \varphi^*(\alpha) - \alpha
$$
which takes the identity to the identity. 

\begin{lemma}
\label{lemma-exactsequence}
The exact sequence (\ref{equn-olssonsequence}) can be continued to an exact sequence
$$
0 \to \operatorname{Pic}_X \to \operatorname{Aut}_{\mathcal{X}} \to \operatorname{Aut}_X \xrightarrow{a} R^2\pi_{ *}\mathbf{G}_m.
$$
of sheaves of pointed sets on $(\operatorname{Sch}/k)_{fppf}$.
\end{lemma}

\begin{proof}
    Exactness on the left is in Olsson's paper, so it suffices to show exactness at $\operatorname{Aut}_X$. Let $\mathcal{Y} \to Y$ be a $\mathbf{G}_m$-gerbe over an algebraic space and $\varphi : Y \to Y$ an automorphism. Then an automorphism $\mathcal{Y} \to \mathcal{Y}$ inducing the identity on bands and inducing $\varphi$ on rigidifications is the same data as a commutative diagram
    $$
    \begin{tikzcd}
        \mathcal{Y} \ar[r] \ar[d] &\mathcal{Y} \ar[d] \\
        Y \ar[r, "\varphi"] &Y,
    \end{tikzcd}
    $$
    which in turn is equivalent to a morphism $\mathcal{Y} \to Y \times _{\varphi, Y} \mathcal{Y}$ over $Y$ inducing the identity on bands, and such a morphism would necessarily be an equivalence of $\mathbf{G}_m$-gerbes over $Y$. This data exists if and only if the pullback of the cohomology class of $\mathcal{Y}$ via $\varphi$ is equal to itself. 

    Applying this to the gerbe $\mathcal{X}_T \to X_T$ for any $k$-scheme $T$ shows that the sequence of pointed sets
    $$
    \mathcal{A}ut_{\mathcal{X}}(T)/\cong \longrightarrow\operatorname{Aut}_X(T) \xrightarrow{\varphi \mapsto \varphi^*\alpha - \alpha}H^2(X_T, \mathbf{G}_m)
    $$
    is exact. Sheafifying with respect to the fppf topology gives the result.
\end{proof}

Let us additionally write $\pi_{\acute{e}t}$ for the structure morphism 
$(\operatorname{Sch}/X)_{\acute{e}t} \to (\operatorname{Sch}/k)_{\acute{e}t}$ of big \'etale sites. For $i \geq 0$ write $\Phi^i, \Phi^i_{fl}$ for the formal completions at the identity of $R^i\pi_{\acute{e}t, *}\mathbf{G}_m, R^i\pi_{ *}\mathbf{G}_m.$ That is, let $\operatorname{Art}_k$ denote the  category of local, finite-dimensional $k$-algebras and for $A \in \operatorname{Art}_k$ let 
\begin{align*}
\Phi^i(A) &= \operatorname{Ker}(R^i\pi_{\acute{e}t, *}\mathbf{G}_m(\operatorname{Spec}A) \to R^i\pi_{\acute{e}t, *}\mathbf{G}_m(\operatorname{Spec}A/\mathfrak{m}_A)), \\ 
\Phi^i_{fl}(A) &= \operatorname{Ker}(R^i\pi_{ *}\mathbf{G}_m(\operatorname{Spec}A) \to R^i\pi_{ *}\mathbf{G}_m(\operatorname{Spec}A/\mathfrak{m}_A)).
\end{align*}
Note that sheaves on $(\operatorname{Sch}/k)_{\tau}$ for $\tau = fppf, \acute{e}t$ have no higher cohomology (over the final object) since $k$ is algebraically closed, so the Leray spectral sequence for $\pi, \pi_{\acute{e}t}$ shows that 
$$
R^i\pi_{\acute{e}t, *}\mathbf{G}_m(\operatorname{Spec}A/\mathfrak{m}_A)) = H^i(X, \mathbf{G}_m) = R^i\pi_{ *}\mathbf{G}_m(\operatorname{Spec}A/\mathfrak{m}_A)).
$$
We have the following results of Raynaud and Ekedahl, building on the work \cite{MR457458} of Artin and Mazur. See also forthcoming work of Bragg and Olsson which uses their results \cite{bragg2025representabilitycohomologyfiniteflat} to generalize the following theorem beyond the case of $\mathbf{G}_m$.

\begin{thm}[Raynaud, Ekedahl]
\label{thm-raynaud}
    \begin{enumerate}
        \item The functors $\Phi^i_{fl}$ are pro-representable. 
        \item The functors $\Phi^i$ are the Artin--Mazur formal groups
        $$
        \Phi^i(A) = \operatorname{Ker}(H^i(X_A, \mathbf{G}_m) \to H^i(X, \mathbf{G}_m)).
        $$
        \item For each $A \in \operatorname{Art}_k$ there is an exact sequence of groups
        $$
        0 \to H^1(\operatorname{Spec}(A)_{fppf}, R^{i-1}\pi_{ *}\mathbf{G}_m) \to \Phi^i(A) \to \Phi^i_{fl}(A) \to 0.
        $$
        \item We have 
        $$
        H^j(\operatorname{Spec}(A)_{fppf}, R^{i}\pi_{ *}\mathbf{G}_m)  = H^j(A, \Phi^{i}_{fl})
        $$
        for all $j > 0$ and $i \geq 0$,
        where on the right hand side, $\Phi^{i}_{fl}$ is viewed as a sheaf on $\operatorname{Art}_k$ with the fppf topology (so coverings are given by finite free ring maps $R \to S$ in $\operatorname{Art}_k$). In particular, both groups are zero when $j > 1$.
    \end{enumerate}
\end{thm}

\begin{proof}
    The key points are (1), which is \cite[Proposition 2.7.5]{MR563468}, and (3), which is equivalent to \cite[III, Proposition 8.1]{MR800174} in view of (4).
    
    To prove (4), first note that for any $A \in \operatorname{Art}_k$, abelian sheaf $\mathcal{F}$ on $(\operatorname{Sch}/k)_{fppf}$, and $i \geq 0$, we can compute
    $$
    H^i(\operatorname{Spec}A, \mathcal{F})
    $$
    in the site $(\operatorname{Art}_k)_{fppf}$ described in (4). This follows in a standard fashion from the fact that an fppf covering can be refined by a quasi-finite flat covering \cite[\href{https://stacks.math.columbia.edu/tag/0572}{Tag 0572}]{stacks-project}, and a quasi-finite flat covering of a local Artin scheme consists of a finite disjoint union of local Artin schemes each finite free over the base. We use this without further mention. Next, note that for $A \in \operatorname{Art}_k$ and $i \geq 0$, the exact sequence
    $$
    0 \to \Phi^i(A) \to R^i\pi_*\mathbf{G}_m(\operatorname{Spec}A) \to H^i(X, \mathbf{G}_m) \to 0
    $$
    (surjectivity holds because $A \to A /\mathfrak{m} = k$ has a section) can be viewed as the $A$-points of an exact sequence of sheaves
    $$
    0 \to \Phi^i \to R^i\pi_* \mathbf{G}_m \to \underline{H^i(X, \mathbf{G}_m)} \to 0
    $$
    on $\operatorname{Art}_k$, where 
    $$\underline{H^i(X, \mathbf{G}_m)} = \coprod _{H^i(X, \mathbf{G}_m)}\operatorname{Spec}(k)$$
    is the constant group scheme on $H^i(X, \mathbf{G}_m)$. This is a smooth group scheme over $k$, so by Grothendieck's Theorem \cite[Th\'eor\`eme 11.7]{MR244271}, it has no higher cohomology over any $A \in \operatorname{Art}_k$ as all \'etale coverings of a local Artin ring with algebraically closed residue field admit a section. The long exact sequence of cohomology associated to this exact sequence of sheaves proves the isomorphism in (4). The proof that the groups $H^j(A, \Phi^{i}_{fl})$ vanish for $j > 1$ is given by Ekedahl in the proof of \cite[III, Proposition 8.1]{MR800174}, and uses Grothendieck's Theorem \emph{loc. cit.} and the fact that a pro-representable formal group is an extension of a formal Lie group by a connected finite group scheme. In view of this vanishing, we note that (3) follows from the Leray spectral sequence for $\pi$.

    Finally, (2) holds because Artin local rings with algebraically closed residue field have no higher \'etale cohomology, so the Leray spectral sequence for $\pi_{\acute{e}t}$ shows that for all $A \in \operatorname{Art}_k$ and $i \geq 0$, 
    $$
    H^0(\operatorname{Spec}A, R^i\pi_{\acute{e}t, *}\mathbf{G}_m) = H^i(X_A, \mathbf{G}_m)
    $$
    from which (2) follows directly. 
\end{proof}

\begin{remark}
\label{remark-vectorspacestructure}
Taking $A = k[\epsilon]$ in (3) gives an exact sequence of $k$-vector spaces
\begin{equation}
\label{equn-badsubspace}
0 \to H^1(\operatorname{Spec}(k[\epsilon], R^{i-1} \pi_*\mathbf{G}_m) \to \Phi^i(k[\epsilon]) \to \Phi^i_{fl}(k[\epsilon]) \to 0
\end{equation}
and we have $\Phi^i(k[\epsilon]) \cong H^i(X, \mathcal{O}_X)$. See \cite[II.4]{MR457458} or the proof of Lemma \ref{lemma-derivative} below.
\end{remark}

Now slightly generalizing the situation at the beginning of the section, fix $i \geq 0$, $\alpha \in H^i(X, \mathbf{G}_m)$ and consider the map 
$$
a: \operatorname{Aut}_{X} \to R^i\pi_{ *}\mathbf{G}_m, \hspace{3 em} \varphi \mapsto \varphi^*\alpha - \alpha.
$$
This is a map of fppf sheaves of sets respecting the identity elements, so it induces a morphism on tangent spaces
\begin{equation}
\label{equn-derivative}
H^0(X, T_X) = T_{\operatorname{id}_X}\operatorname{Aut}_X = \operatorname{Ker}(\operatorname{Aut}_X(k[\epsilon]) \to \operatorname{Aut}_X(k)) \to \Phi^i_{fl}(k[\epsilon])
\end{equation}
which we will now compute. Let us write $\operatorname{dlog}(\alpha) \in H^i(X, \Omega^1_X)$ for the image of $\alpha$ under the map on \'etale cohomology coming from the morphism of abelian sheaves
$$
\mathbf{G}_m \to \Omega^1, \hspace{3 em} u \mapsto \operatorname{dlog}(u) = du/u
$$
on the small \'etale site $X_{\acute{e}t}$ of $X$. Note that we use Grothendieck's Theorem to identify $H^i(X, \mathbf{G}_m)$ with \'etale cohomology here. 

\begin{lemma}
\label{lemma-derivative}
    The tangent map (\ref{equn-derivative}) factors as
    $$
    H^0(X, T_X)\xrightarrow{\theta \mapsto \theta (\operatorname{dlog}(\alpha))} H^i(X, \mathcal{O}_X) = \Phi^i(k[\epsilon]) \xrightarrow{b} \Phi^i_{fl}(k[\epsilon])
    $$
    where $b$ is the map from (\ref{equn-badsubspace}).  
\end{lemma}

\begin{proof}
    There are canonical bijections between (a) elements in $H^0(X, T_X)$ viewed as $\mathcal{O}_X$-linear maps $\Omega_X \to \mathcal{O}_X$, (b) automorphisms $X[\epsilon] \to X[\epsilon]$ over $k[\epsilon]$ which equal the identity on the special fiber, and (c)  $k$-linear derivations $\mathcal{O}_X \to \mathcal{O}_X$. Let $\theta \in H^0(X, T_X)$ correspond to $\varphi : X[\epsilon] \to X[\epsilon]$ and $D : \mathcal{O}_X \to \mathcal{O}_X$. In formulas, 
    \begin{equation}
        \label{equn-therelations}
    \varphi^\#(f+g\epsilon) = f + (g+D(f))\epsilon, \hspace{3 em} \theta(df) = D(f),
    \end{equation}
    where $f,g$ are local sections of $\mathcal{O}_X$ and $\varphi^\#$ is the automorphism of sheaves $\mathcal{O}_{X[\epsilon]} \to \mathcal{O}_{X[\epsilon]}$ corresponding to $\varphi$. Then by definition, the map (\ref{equn-derivative}) takes $\theta$ to $\varphi ^*\alpha - \alpha \in \Phi^i_{fl}(k[\epsilon])$. This factors as
    $$
    H^0(X, T_X) \xrightarrow{\theta \mapsto \varphi^*\alpha - \alpha} \Phi^i(k[\epsilon]) \to \Phi^i_{fl}(k[\epsilon]),
    $$
    where the second map is (\ref{equn-badsubspace}), so it suffices to compute the first map. 

    One computes $\Phi^i(k[\epsilon])$ using the truncated exponential sequence
    $$
    0 \to \mathcal{O}_X \xrightarrow{\operatorname{exp}} \mathcal{O}_{X[\epsilon]}^\times \to \mathcal{O}_X^\times \to 0, 
    $$
    which we may view as a sequence of sheaves on $X_{\acute{e}t}$. 
    As the sequence is split by the inclusion $j : \mathcal{O}_X^\times \to \mathcal{O}_{X[\epsilon]}^\times$, we obtain for all $i \geq 0$ short exact sequences
    $$
    0 \to H^i(X, \mathcal{O}_X) \to H^i(X[\epsilon], \mathbf{G}_m) \to H^i(X, \mathbf{G}_m) \to 0
    $$
    and hence isomorphisms $H^i(X, \mathcal{O}_X) \cong \Phi^i(k[\epsilon])$. In view of this, the formula we must prove is $\varphi^*j\xi -j\xi = \operatorname{exp}(\theta(\operatorname{dlog}(\xi)))$. For this, it suffices to show that $(\varphi^\sharp \circ j) \cdot (1/j) = \operatorname{exp}\circ \theta\circ \operatorname{dlog}$ as maps of sheaves $\mathcal{O}_X^\times \to \mathcal{O}_{X[\epsilon]}$ (note that while the operation in $\mathbf{G}_m$ is written multiplicatively, the operation in $H^i(-, \mathbf{G}_m)$ is written additively).

    Given a local section $u$ of $\mathcal{O}_X^\times$, we  have 
    $$
    \varphi^\# j(u) \cdot j(u)^{-1} = (u + D(u) \epsilon ) \cdot u^{-1} = 1+ u^{-1}D(u) \epsilon
    $$
    by the relations (\ref{equn-therelations}), while
    $$
    \operatorname{exp}(\theta (\operatorname{dlog}(u))) = \operatorname{exp}(u^{-1}\theta(du)) = \operatorname{exp}(u^{-1}\cdot D(u)) = 1 + u^{-1}D(u)\epsilon ,
    $$
    as needed.
\end{proof}

\section{Positive results}
\label{section-sufficient}

Let $X$ be a smooth proper variety over an algebraically closed field $k$ of characteristic $p>0$. Let $\mathcal{X} \to X$ be a $\mathbf{G}_m$-gerbe with cohomology class $\alpha \in H^2(X, \mathbf{G}_m)$. Denote by $\pi$ the structure morphism $(\operatorname{Sch}/X)_{fppf} \to (\operatorname{Sch}/k)_{fppf}$. Then by \cite[Proposition 1.4]{MR1608805}, $\alpha$ is torsion. Let $n$ be the least positive integer such that $n \alpha = 0$. Write $n = p^i m$ with $p \nmid m$. Then we can uniquely write 
$$
\alpha = \alpha_1 + \alpha_2
$$
where $p^i\alpha_1 = 0$ and $m \alpha_2 = 0$. 

\begin{lemma}
\label{lemma-primetoporpower}
    The morphism $a : \operatorname{Aut}_X \to R^2\pi_*\mathbf{G}_m,$ $\varphi \mapsto \varphi^*\alpha - \alpha$, factors 
    $$
    \operatorname{Aut}_X \to R^2 \pi_* \mathbf{G}_m[p^i] \oplus R^2\pi_*\mathbf{G}_m[m] = R^2\pi_*\mathbf{G}_m[n] \hookrightarrow R^2\pi_*\mathbf{G}_m,
    $$
    where the map $\operatorname{Aut}_X \to R^2 \pi_* \mathbf{G}_m[p^i]$ is given by $\varphi \mapsto \varphi^*\alpha_1 - \alpha_1$ and the map $\operatorname{Aut}_X \to  R^2\pi_*\mathbf{G}_m[m]$ is given by $\varphi \mapsto \varphi^*\alpha_2 - \alpha_2$.\qed
\end{lemma}

This allows us to treat the cases $n = p^i$ and $p \nmid n$ separately. Olsson shows \cite[Theorem 1.1(3)]{MR4961251} that when $p \nmid n$, the restriction $a|_{\operatorname{Aut}^0_X} : \operatorname{Aut}^0_X \to R^2\pi_*\mathbf{G}_m$ is zero: It factors through the \'etale $k$-scheme $R^2\pi_*\mu_n$ and therefore must be constant since $\operatorname{Aut}^0_X$ is connected. Thus we are reduced to understanding the case $n = p^i$. 

\begin{thm}
\label{thm-sufficientconditions}
Assume any of the following hold:
\begin{enumerate}
    \item  $H^0(X, T_X) = 0$,
    \item $H^2(X, \mathcal{O}_X) = 0$, or more generally $H^1(k[\epsilon], \operatorname{Pic}_X) = H^2(X, \mathcal{O}_X)$, 
    \item The commutative $k$-group scheme $R^2\pi_*\mu_{p^i}$ is of multiplicative type ($p^i$ as above), or
    \item $\operatorname{Aut}^0_X$ is an abelian variety.
\end{enumerate}
Then the restriction $a|_{\operatorname{Aut}^0_X} : \operatorname{Aut}^0_X \to R^2 \pi_*\mathbf{G}_m$ is zero and so 
$
\operatorname{Aut}^0_{\mathcal{X}} \to \operatorname{Aut}^0_X
$
is surjective.
\end{thm}

\begin{examples}
\begin{enumerate}
    \item If $X$ is a K3 surface, then (1) holds so $\operatorname{Aut}^0_{\mathcal{X}} \to \operatorname{Aut}^0_X$ is surjective for any $\mathbf{G}_m$-gerbe $\mathcal{X} \to X$.
    \item If $X$ is an abelian variety, then (4) holds, so $\operatorname{Aut}^0_{\mathcal{X}} \to \operatorname{Aut}^0_X$ is surjective for any $\mathbf{G}_m$-gerbe $\mathcal{X} \to X$.
    \item When $k$ has characteristic 2 and $X$ is a classical Enriques surface over $k$, then $H^2(X, \mathbf{G}_m) \cong \mathbf{Z}/2$, and if $\alpha \in H^2(X, \mathbf{G}_m)$ is the non-zero class, then $0 \neq \operatorname{dlog}(\alpha) \in H^2(X, \Omega^1_X)$. However, $H^2(X, \mathcal{O}_X) = 0$, so $\operatorname{Aut}^0_{\mathcal{X}} \to \operatorname{Aut}^0_X$ is surjective for the $\mathbf{G}_m$-gerbe $\mathcal{X} \to X$ with class $\alpha$. See \cite[Section 9]{Antieau_Bhatt_Mathew_2021}.
    \end{enumerate}
\end{examples}

\begin{proof}
    The second half of the conclusion follows from the first by Lemma \ref{lemma-exactsequence}. Thus we need only prove $a|_{\operatorname{Aut}^0_X}$ is zero. By Lemma \ref{lemma-primetoporpower}, \cite[Theorem 1.1(3)]{MR4961251}, and the discussion above, we may assume the order $n$ of $\alpha$ satisfies $n = p^i$ for some integer $i > 0$. 
    
    If (1) holds, then $\operatorname{Aut}^0_X = 1$ is trivial, so certainly the result is true.

    Suppose (2) holds. Consider the Kummer exact sequence
    $$
    0 \to \operatorname{Pic}_X/p^i\operatorname{Pic}_X \to R^2\pi_*\mu_{p^i} \to R^2\pi_*\mathbf{G}_m[p^i] \to 0.
    $$
    By \cite[Corollary 1.6]{bragg2025representabilitycohomologyfiniteflat}, the sheaf $R^2\pi_*\mu_{p^i}$ is representable by an affine group scheme of finite type over $k$. Also, the group scheme $\operatorname{Pic}_X /p^i\operatorname{Pic}_X$ is finite over $k$: It suffices to show it has only finitely many $k$-points, which is true by the Theorem of the Base (note that $\operatorname{Pic}^0_{X, red}$ is an abelian variety so multiplication by $p^i$ is surjective on $\operatorname{Pic}^0_X(k)$). Thus $R^2\pi_*\mathbf{G}_m[p^i]$ is representable by an affine group scheme of finite type over $k$, being a quotient of an affine group scheme by a normal closed subgroup scheme. Furthermore, the Lie algebra of $R^2\pi_*\mathbf{G}_m[p^i]$ is equal to the Lie algebra of $R^2 \pi_*\mathbf{G}_m$, that is, $\Phi^2_{fl}(k[\epsilon]),$ see Section \ref{section-prelims}. This is because multiplication by $p$ acts trivially on the Lie algebra of a prorepresentable formal group. But the hypotheses imply that $\Phi^2_{fl}(k[\epsilon]) = 0$ by Theorem \ref{thm-raynaud}. Since $R^2\pi_*\mathbf{G}_m[p^i]$ is a finite type group scheme over $k$ whose Lie algebra is trivial, its connected component of the identity is equal to $\operatorname{Spec}(k)$, and so $a|_{\operatorname{Aut}^0_X}$ is equal to zero, as needed. 

    For (3) or (4), by the Kummer sequence, there is $\beta \in H^2(X, \mu_{p^i})$ mapping to $\alpha$, and then the map $a|_{\operatorname{Aut}^0_X}$ factors through the map
    $$
    \operatorname{Aut}^0_X \to R^2\pi_*\mu_{p^i}, \hspace{3 em} \varphi \mapsto \varphi^*\beta - \beta.
    $$
    Again by \cite[Corollary 1.6]{bragg2025representabilitycohomologyfiniteflat}, the sheaf $R^2\pi_*\mu_{p^i}$ is representable by an affine group scheme of finite type over $k$. If (3) holds, then the action 
    $$
    \operatorname{Aut}^0_X \to \underline{\operatorname{Hom}}(R^2\pi_*\mu_{p^i}, R^2\pi_*\mu_{p^i})
    $$
    is a map from a connected group scheme to an \'etale group scheme (using that $R^2\pi_*\mu_{p^i}$ is of multiplicative type) and so is constant, proving the result in this case. Finally, if (4) holds, then $\operatorname{Aut}^0_X \to R^2\pi_*\mu_{p^i}$ is a morphism from an abelian variety over $k$ to an affine $k$-scheme, and hence factors through $\operatorname{Spec}(k)$, proving the result.
    \end{proof}

\begin{corollary}
\label{corollary-applicationtoabelianvarieties}
    Suppose $X, Y$ are abelian varieties over $k$, $\alpha \in \operatorname{Br}(X), \beta \in \operatorname{Br}(Y)$, and there is a Fourier--Mukai equivalence $D^b_{\operatorname{Coh}}(X, \alpha) \cong D^b_{\operatorname{Coh}}(Y, \beta)$ (see \cite[Definition 1.3]{MR4961251}). Then $X$ and $Y$ are isogenous.
\end{corollary}

\begin{proof}
Let $\mathcal{X} \to X, \mathcal{Y} \to Y$ denote the $\mathbf{G}_m$-gerbes corresponding to $\alpha, \beta$. By \cite[Theorem 1.9]{MR4961251}, there is an isomorphism of $k$-group algebraic spaces $\operatorname{Aut}^0_{\mathcal{X}} \cong \operatorname{Aut}^0_{\mathcal{Y}}$. We claim that the sequences
$$
0 \to \operatorname{Pic}^0_X \to \operatorname{Aut}^0_{\mathcal{X}} \to \operatorname{Aut}^0_X \to 0
$$
$$
0 \to \operatorname{Pic}^0_Y \to \operatorname{Aut}^0_{\mathcal{Y}} \to \operatorname{Aut}^0_Y \to 0
$$
coming from Olsson's sequence (\ref{equn-olssonsequence}) are exact. Exactness on the right is by Theorem \ref{thm-sufficientconditions}(4). Exactness on the left is essentially a consequence of the fact that the Neron--Severi group of an abelian variety is torsion free. Write $G = \operatorname{Aut}^0_{\mathcal{X}} \cap \operatorname{Pic}_X$ for the kernel of $\operatorname{Aut}^0_{\mathcal{X}} \to \operatorname{Aut}^0_X$  and consider the inclusions of group algebraic spaces
$$
\begin{tikzcd}
   \operatorname{Pic}^0_X \ar[r, hook] &G \ar[r, hook] \ar[d, hook] &\operatorname{Aut}^0_{\mathcal{X}} \ar[d, hook] \\
    &\operatorname{Pic}_X \ar[r, hook] & \operatorname{Aut}_{\mathcal{X}}.
\end{tikzcd}
$$
Then $G$ is quasi-compact being a closed subgroup space of $\operatorname{Aut}^0_{\mathcal{X}}$ and $G/\operatorname{Pic}^0_X$ is \'etale since $G$ is a closed and open subgroup scheme of $\operatorname{Pic}_X$. Thus $G/\operatorname{Pic}^0_X$ being a quasi-compact closed subgroup scheme of the \'etale group scheme $\operatorname{Pic}_X /\operatorname{Pic}^0_X$ is finite \'etale, and thus trivial since $\operatorname{NS}(X)$ is torsion free. The same arguments of course also apply to $Y$.

Now by Poincar\'e's Complete Reducibility Theorem, $\operatorname{Aut}^0_{\mathcal{X}}$ (resp. $\operatorname{Aut}^0_{\mathcal{Y}}$) is isogenous to $\operatorname{Aut}^0_X \times \operatorname{Pic}^0_X = X \times X^t$ (resp. $Y \times Y^t$). Since an abelian variety is isogenous to its dual, we obtain that $X \times X$ and $Y \times Y$ are isogenous, and then using Poincar\'e's Complete Reducibility Theorem again to decompose $X, Y$ into products of simple abelian varieties (up to isogeny), we see that $X$ and $Y$ are isogenous. 
\end{proof}

\section{Aside: Complete and torsion groups}
\label{section-completetorsion}

This section is used only in the proof of Theorem \ref{lemma-cofinite} which is not needed for the proofs of the other main results of the paper, so can be safely skipped by most readers. We recall here a special case of a result sometimes known as Greenlees--May duality. Fix a prime number $p>0$. Our goal is to give a criterion for a $p$-power torsion abelian group to be isomorphic to $(\mathbf{Q}_p/\mathbf{Z}_p)^{\oplus e} \oplus F$ where $e\geq 0$ is an integer and $F$ is a finite abelian $p$-group. Applying this to the $p$-power torsion subgroup of $H^i(S, \mathbf{G}_m)$ for a scheme $S$, we obtain a generalization of a result of \cite{MR4936530} on the $p$-power torsion subgroup of the Brauer group of an ordinary variety, see Theorems \ref{prop-cofinite} and \ref{lemma-cofinite}. 

For an abelian group $A$, write $A[p^{\infty}] = \bigcup_{n \geq 1} A[p^n]$ for the $p$-power torsion subgroup of $A$. We say $A$ is $p$-power torsion if $A = A[p^{\infty}]$. Denote by $D_{p^{\infty}-tors}(\mathbf{Z})$ the full subcategory of $D(\mathbf{Z})$ (the unbounded derived category of the category of abelian groups) whose objects $K$ have the property that $H^i(K)$ is $p$-power torsion for all $i \in \mathbf{Z}$. Thus $K \in D_{p^{\infty}-tors}(\mathbf{Z})$ if and only if $K \otimes^{\mathbf{L}} \mathbf{Z}[1/p] = 0$. The inclusion $D_{p^{\infty}-tors}(\mathbf{Z}) \to D(\mathbf{Z})$ admits a right adjoint denoted $R \Gamma_p$ and defined by $R\Gamma_p(K) = K \otimes^{\mathbf{L}}\operatorname{Cone}(\mathbf{Z} \to \mathbf{Z}[1/p])[-1]$. See \cite[\href{https://stacks.math.columbia.edu/tag/0A6R}{Tag 0A6R}]{stacks-project}.

An object $K \in D(\mathbf{Z})$ is called \emph{derived $p$-complete} if $R\operatorname{Hom}_{\mathbf{Z}}(\mathbf{Z}[1/p], K) = 0$. We denote by $D_{p-comp}(\mathbf{Z})$ the full subcategory consisting of derived $p$-complete objects. The inclusion $D_{p-comp}(\mathbf{Z})\to D(\mathbf{Z})$ admits a left 
adjoint denoted $(-)^\wedge$ where $K^\wedge = R \operatorname{Hom}(\operatorname{Cone}(\mathbf{Z} \to \mathbf{Z}[1/p]), K)$. See \cite[\href{https://stacks.math.columbia.edu/tag/091V}{Tag 091V}]{stacks-project}. This functor is called \emph{derived $p$-completion}. If $K = A[0]$ with $A$ an abelian group, then we have $H^{-1}(K^\wedge) = T_pA$ where $T_pA = \operatorname{lim}_{n}A[p^n]$ is the $p$-adic Tate module, there is a short exact sequence
$$
0 \to \operatorname{lim}^1 _{n}A[p^n] \to H^0(K^\wedge) \to \widehat{A} \to 0 
$$
where $\widehat{A} = \operatorname{lim}_n A/p^nA$ is the ordinary $p$-adic completion, and $H^i(K^\wedge) = 0$ for $i \neq -1, 0$. See \cite[\href{https://stacks.math.columbia.edu/tag/0BKG}{Tag 0BKG}]{stacks-project}. It is also true that for $K \in D(\mathbf{Z})$, we have $K \in D_{p-comp}(\mathbf{Z}) \iff H^i(K) \in D_{p-comp}(\mathbf{Z})$ for all $i \in \mathbf{Z}$, see \cite[\href{https://stacks.math.columbia.edu/tag/091P}{Tag 091P}]{stacks-project}.

The following is a very special case of the results \cite{Dwyer-Greenlees}, \cite{Porta-Liran-Yekutieli}, {\cite[\href{https://stacks.math.columbia.edu/tag/0A6X}{Tag 0A6X}]{stacks-project}}. 

\begin{thm}[Dwyer, Greenlees, May]
\label{thm-gmduality}
    The functors $R \Gamma_p : D(\mathbf{Z}) \to D_{p^{\infty}-tors}(\mathbf{Z})$ and $(-)^\wedge : D(\mathbf{Z}) \to D_{p-comp}(\mathbf{Z})$ induce mutually quasi-inverse equivalences of categories 
    $$
    R \Gamma_p : D_{p-comp}(\mathbf{Z}) \longleftrightarrow D_{p^{\infty}-tors}(\mathbf{Z}) : (-)^\wedge.
    $$
\end{thm}

\begin{example}
\label{example-dualityforzp}
    One computes directly that $R\Gamma_p(\mathbf{Z}_p) = \mathbf{Q}_p/\mathbf{Z}_p[-1]$, and $(\mathbf{Q}_p/\mathbf{Z}_p)^\wedge = T_p(\mathbf{Q}_p/\mathbf{Z}_p)[1] = \mathbf{Z}_p[1]$. If $F$ is a finitely generated, torsion $\mathbf{Z}_p$-module (i.e. a finite abelian $p$-group) then $F^\wedge \cong F$ and $R\Gamma_p(F) \cong F$. 
\end{example}

The following is the promised characterization.

\begin{lemma}
\label{lemma-characterization}
    Let $A$ be a $p$-power torsion abelian group. The following are equivalent.
\begin{enumerate}
    \item There is an integer $e \geq 0$ and a finite abelian $p$-group $F$ such that $A \cong (\mathbf{Q}_p/\mathbf{Z}_p)^{\oplus e} \oplus F$. 
    \item The groups $A[p]$ and $A/pA$ are finite. 
    \item $T_pA$ is a finite free $\mathbf{Z}_p$-module, $\widehat{A}$ is a finite $\mathbf{Z}_p$-module, and the system $\{A[p^n]\}_{n \geq 1}$ satisfies the Mittag--Leffler condition (ML).
\end{enumerate}
\end{lemma}

\begin{proof}
    The implication (1) $\implies$ (2) can be checked directly. Assume (2) holds. One proves by induction that the groups $A[p^n]$ are finite for all $n \geq 1$ using the exact sequences
    $$
    0 \to A[p] \to A[p^{n}] \xrightarrow{p} A[p^{n-1}].
    $$
    Then it follows that the system $\{A[p^n]\}_{n \geq 1}$ satisfies ML. Furthermore, $T_pA$ is $p$-torsion free and $p$-adically complete (and separated in Bourbaki terminology): For each $n \geq 1$ there is a short exact sequence
    $$
    0 \to T_pA \xrightarrow{p^n} T_pA \to A[p^n] \to 0,
    $$
    from which one sees that the system $T_pA/p^nT_pA$ is isomorphic to the system $A[p^n]$ which has limit $T_pA$. Thus by the well-known criterion \cite[\href{https://stacks.math.columbia.edu/tag/031D}{Tag 031D}]{stacks-project}, to show that $T_pA$ is a finite free $\mathbf{Z}_p$-module, it suffices to show $T_pA/pT_pA \cong A[p]$ is finite, which follows from the assumptions. By the same criterion, the fact that $A/pA$ is finite implies that $\widehat{A}$ is a finite $\mathbf{Z}_p$-module.

    Finally assume (3) holds. Then using \cite[\href{https://stacks.math.columbia.edu/tag/0BKG}{Tag 0BKG}]{stacks-project} discussed above we compute
    $$
    A^\wedge \cong  T_pA [1] \oplus \widehat{A} \cong \mathbf{Z}_p^{\oplus e}[1] \oplus \mathbf{Z}_p^{\oplus e'} \oplus F
    $$
    for some $e, e' \geq 0$ and finite abelian $p$-group $F$ (where we use that every object in $D(\mathbf{Z})$ is isomorphic to the direct sum of its cohomology groups appropriately shifted, and every finite $\mathbf{Z}_p$-module is a direct sum of a free module and a torsion module). By Theorem \ref{thm-gmduality} and Example \ref{example-dualityforzp} we have
    $$
    A \cong R\Gamma_p(A^\wedge) = (\mathbf{Q}_p/\mathbf{Z}_p)^{\oplus e} \oplus F \oplus (\mathbf{Q}_p/\mathbf{Z}_p)^{\oplus e'}[-1].
    $$
    This implies $e' = 0$. It follows that (1) holds. 
\end{proof}

\begin{corollary}
\label{cor-tocriterion}
    Let $A$ satisfy the equivalent conditions of Lemma \ref{lemma-characterization}. Then $e$ is the rank of the finite free $\mathbf{Z}_p$-module $T_pA$ and $F \cong \widehat{A}$. There is also an isomorphism
    $$
    \operatorname{Coker}(T_pA \to A[p]) \cong F[p].
    $$
\end{corollary}

\begin{proof}
    The first part follows directly from the proof of Lemma \ref{lemma-characterization}. For the second part, write $A = A_1 \oplus A_2$ where $A_1 = (\mathbf{Q}_p/\mathbf{Z}_p)^{\oplus e}$ and $A_2 = F$. Then $T_pA_1 \to A_1[p]$ is surjective, $T_pA_2 = 0$, from which the next part follows.
\end{proof}

We end with the application we have in mind.

\begin{thm}
\label{prop-cofinite}
    Let $S$ be a scheme and $i \geq 0$ an integer. Assume $H^i(S, \mu_p)$ and $H^{i+1}(S, \mu_p)$ are finite groups. Then there are an integer $e \geq 0$ and a finite abelian $p$-group $F$ such that
    $$
    H^i(S, \mathbf{G}_m)[p^{\infty}] \cong (\mathbf{Q}_p/\mathbf{Z}_p)^{\oplus e} \oplus F.
    $$
\end{thm}

\begin{proof}
    We use criterion (2) of Lemma \ref{lemma-characterization}. 
    By the Kummer exact sequences
    $$
    0 \to H^{j-1}(S, \mathbf{G}_m)/p \to H^j(S, \mu_{p}) \to H^{j}(S, \mathbf{G}_m)[p] \to 0
    $$
    we see that both $H^{i}(S, \mathbf{G}_m)/p$ and $H^{i}(S, \mathbf{G}_m)[p]$  are finite. We conclude that
    $(H^{i}(S, \mathbf{G}_m)[p^\infty])\otimes _{\mathbf{Z}}\mathbf{Z}/p$ is finite by the lemma below and the result follows from Lemma \ref{lemma-characterization}.
\end{proof}

\begin{lemma}
    Let $A$ be an abelian group. Set $B = A[p^\infty]$. If $A/pA$ is finite, then so is $B/pB$.
\end{lemma}

\begin{proof}
    Apply the snake lemma to the endomorphism ``multiplication by $p$'' of the short exact sequence
    $$
    0 \to B \to A \to A/B \to 0.
    $$
    Since $A/B$ is $p$-torsion free we see that $B/pB \hookrightarrow A/pA$ and the result follows. 
\end{proof}

\section{Flat cohomology of an ordinary variety}
\label{section-flatcohomologyordinary}

In this section, let $X$ be a smooth proper scheme over an algebraically closed field of positive characteristic $p$. Recall that $X$ is \emph{ordinary} (see \cite{MR699058} where the authors attribute this notion to Kato) if $H^i(X, B^j) = 0$ for all $i \geq 0$ and $j > 0$ where $B^j = B^j_X = \operatorname{Im}(\Omega^{j-1}_X \xrightarrow{d^{j-1}} \Omega^j_X)$ is the image of the de Rham differential. An easy consequence of the definition (which we will not need) is that the Hodge--de Rham spectral sequence for such an $X$ degenerates at $E_1$.

Similar arguments to the proof of the following lemma can be found in \cite{MR916481}.

\begin{lemma}
\label{lemma-computation}
    Assume $X$ is ordinary. For $i \geq 0$, we have $H^{i+1}(X, \mu_p) \cong (\mathbf{Z}/p)^n$ where $n = \operatorname{dim}_k(H^{i}(X, \Omega^1_X))$.
\end{lemma}

Note that by the Kummer sequence, $H^0(X, \mu_p) = 0$.

\begin{proof}
Writing $Z^1 = Z^1_X = \operatorname{Ker}(\Omega^1_X \xrightarrow{d} \Omega^2_X)$ we have an exact sequence from the Cartier isomorphism
$$
0 \to B^1 \to Z^1 \xrightarrow{C} \Omega^1_X \to 0
$$
and so $C$ induces bijections $C: H^i(X, Z^1) \xrightarrow{\cong} H^i(X, \Omega^1_X)$ for all $i \geq 0$ by the assumption. We also have an exact sequence
$$
0 \to Z^1 \to \Omega^1 \xrightarrow{d} B^2 \to 0
$$
and it follows that the inclusion $j : Z^1 \to \Omega^1_X$ induces isomorphisms $j: H^i(X, Z^1) \xrightarrow{\cong} H^i(X, \Omega^1_X)$ for all $i \geq 0$. \
Recall Milne's exact (for the \'etale topology) sequence
$$
0 \to \nu_p \to Z^1 \xrightarrow{j - C} \Omega^1 \to 0
$$
of sheaves on $X_{\acute{e}t}$ \cite[Lemma 1.3]{MR460331}, where $\nu_p = \mathbf{G}_m/(\mathbf{G}_m^p)$, and  so
$$
H^i(X_{\acute{e}t}, \nu_p) = H^{i+1}(X, \mu_p),
$$
and the composition $\mathbf{G}_m \to \nu_p \to Z^1$ is the $\operatorname{dlog}$ map.

Since $j : H^i(X, Z^1) \to H^i(X, \Omega^1_X)$ is an isomorphism for $i \geq 0$, we can identify the map $H^i(X, Z^1) \xrightarrow{j-C}H^i(X, \Omega^1_X)$ with the map $H^i(X, \Omega^1_X) \xrightarrow{1 - Cj^{-1}}H^i(X, \Omega^1_X)$ and then we conclude by the following well-known lemma.
\end{proof}

\begin{lemma}
   Let $V$ be a finite-dimensional vector space over $k$. Let $T : V \to V$ be an additive map such that $T(av) = a^{1/p}T(v)$ for $a \in k, v \in V$. 
   Then:
   \begin{enumerate}
       \item $1 - T$ is surjective.
       \item If $T : V \to V$ is bijective, then the $\mathbf{F}_p$-vector space $\operatorname{Ker}(1 - T : V \to V)$ has dimension $= \operatorname{dim}_kV$ and moreover
       $$
       \operatorname{Ker}(1 - T : V \to V) \otimes_{\mathbf{F}_p} k = V
       $$
   \end{enumerate}
\end{lemma}

\begin{remark}
\label{remark-extensionofscalars}
    The proof shows that in fact if $\varphi: \nu_p \to \Omega^1_X$ is the composition of $\nu_p \to Z^1 \xrightarrow{j} \Omega^1_X$, then 
    $$
    \varphi: H^{i+1}(X, \mu_p) = H^i(X, \nu_p) \to H^i(X, \Omega^1_X)
    $$
    is injective and its image $I$ satisfies
    $$
    I \otimes_{\mathbf{F}_p} k = H^i(X, \Omega^1_X). 
    $$
\end{remark}

Illusie defines in \cite[II, n\textsuperscript{o} 5]{MR565469} 
$$
H^i(X, \mathbf{Z}_p(1)) = \operatorname{lim}_n H^i(X, \mu_{p^n}),
$$
and argues that for all $i \geq 0$ the system above satisfies ML and $H^i(X, \mathbf{Z}_p(1))$ is $p$-adically complete (and separated in Bourbaki terminology). 

There is a universal coefficients theorem for these groups. For each $n>0$, there is a short exact sequence of group schemes
\begin{equation}
\label{equn-induction}
0 \to \mu_{p^{n-1}} \to \mu_{p^n} \to \mu_p \to 0
\end{equation}
which gives for every $n$ a long exact sequence of cohomology. Viewing these as systems over $n$ and taking a limit, and using the fact that the systems satisfy ML, we get a long exact sequence of cohomology
$$
\cdots \to H^i(X, \mathbf{Z}_p(1)) \to H^i(X, \mathbf{Z}_p(1)) \to H^i(X, \mu_p) \to H^{i+1}(X, \mathbf{Z}_p(1)) \to \cdots.
$$
In fact, the morphisms $H^i(X, \mathbf{Z}_p(1)) \to H^i(X, \mathbf{Z}_p(1))$ are multiplication by $p$: In general if $A$ is a sheaf of abelian groups on a site, then multiplication by $p$ on $T_pA$ is injective with image $\operatorname{lim}_n A[p^{n-1}]$. We have proven:

\begin{lemma}
\label{lemma-univcoeff}
    For all $i \geq 0$ there are short exact sequences
    $$
  0 \to H^i(X, \mathbf{Z}_p(1))/p  \to H^i(X, \mu_p) \to H^{i+1}(X, \mathbf{Z}_p(1))[p] \to 0. 
    $$
    \qed
\end{lemma}

Illusie also shows the following. 

\begin{lemma}
\label{lemma-kummerpadics}
The Kummer sequences induce exact sequences 
$$
0 \to \widehat{H^i(X, \mathbf{G}_m)} \to H^i(X, \mathbf{Z}_p(1)) \to T_pH^{i+1}(X, \mathbf{G}_m) \to 0.
$$
\end{lemma}

Illusie defines $H^i(X, \mathbf{Q}_p(1)) := H^i(X, \mathbf{Z}_p(1))[1/p]$, and proves a fundamental relation between this group and crystalline cohomology.

\begin{thm}[{\cite[II, Th\'eor\`eme 5.5]{MR565469}}]
    For $i \geq 0$,
$$
H^i(X, \mathbf{Q}_p(1))= \operatorname{Ker}(F-p : H^{i}_{crys}(X/W)_K \to H^{i}_{crys}(X/W)_K)
$$
where $W = W(k)$ and $K = \operatorname{Frac}(W)$. In particular, its dimension over $\mathbf{Q}_p$ is the dimension over $K$ of the part of $H^i(X/W)_K$ with slope exactly one.
\end{thm}

\begin{lemma}
\label{lemma-finitetypeifordinary}
    Assume $X$ is ordinary. Then $H^i(X, \mathbf{Z}_p(1))$ is a finite $\mathbf{Z}_p$-module for all $i$.
\end{lemma}

Illusie proved in \emph{loc. cit.} that $H^i(X, \mathbf{Z}_p(1))$ is a finite $\mathbf{Z}_p$-module for $i = 1, 2$ for any smooth proper variety over $k$. It is easily seen to be zero when $i = 0$ and torsion-free when $i = 1$. He also showed that $H^3(X, \mathbf{Z}_p(1))\cong k$ when $X$ is a supersingular $K3$ surface. 

\begin{proof}
    It is $p$-adically complete (and separated) so it suffices to show $H^i(X, \mathbf{Z}_p(1))/p$ has finite dimension as an $\mathbf{F}_p$-vector space, which follows from Lemmas \ref{lemma-computation} and \ref{lemma-univcoeff}.
\end{proof}

The following is a slight generalization of \cite[Theorem 1.1]{MR4936530}. 

\begin{thm}
\label{lemma-cofinite}
Assume $X$ is ordinary. Then for all $i \geq 1$ there are an integer $e \geq 0$ and a finite Abelian $p$-group $F$ such that 
    $$
    H^i(X, \mathbf{G}_m)[p^{\infty}] \cong (\mathbf{Q}_p/\mathbf{Z}_p)^{\oplus e} \oplus F.
    $$
\end{thm}

Note that $H^0(X, \mathbf{G}_m)[p^\infty] = (k^\times )[p^\infty] = 0$.

\begin{proof}
    This follows from Theorem \ref{prop-cofinite} and Lemma \ref{lemma-computation}.
\end{proof}

\begin{lemma}
\label{lemma-recoveref}
    In the setting of Theorem \ref{lemma-cofinite}, we can recover $e$ via
    $$
    T_pH^{i}(X, \mathbf{G}_m) \cong \mathbf{Z}_p^{\oplus e}
    $$
    and the $p$-torsion of $F$ via
    $$
    F[p] \cong H^{i+1}(X, \mathbf{Z}_p(1))[p].
    $$
\end{lemma}

\begin{proof}
    We have
    $$
    T_pH^{i}(X, \mathbf{G}_m) = T_p(H^{i}(X, \mathbf{G}_m)[p^{\infty}]) = T_p((\mathbf{Q}_p/\mathbf{Z}_p)^{\oplus e} \oplus F) \cong  \mathbf{Z}_p^{\oplus e}.
    $$
    Next, there is a commutative diagram with exact rows
\[\begin{tikzcd}
	0 & {\widehat{H^{i-1}(X, \mathbf{G}_m))}} & {H^{i}(X, \mathbf{Z}_p(1))} & {T_pH^i(X, \mathbf{G}_m)} & 0 \\
	0 & {H^{i-1}(X, \mathbf{G}_m)/p} & {H^{i}(X, \mu_p)} & {H^{i}(X, \mathbf{G}_m)[p]} & 0
	\arrow[from=1-1, to=1-2]
	\arrow[from=1-2, to=1-3]
	\arrow[from=1-2, to=2-2]
	\arrow[from=1-3, to=1-4]
	\arrow[ from=1-3, to=2-3]
	\arrow[from=1-4, to=1-5]
	\arrow[ from=1-4, to=2-4]
	\arrow[from=2-1, to=2-2]
	\arrow[from=2-2, to=2-3]
	\arrow[from=2-3, to=2-4]
	\arrow[from=2-4, to=2-5].
\end{tikzcd}\]
The left vertical arrow is surjective, so the snake lemma identifies the cokerenels of the middle and right vertical arrows. By universal coefficients (Lemma \ref{lemma-univcoeff}), the cokernel of the middle arrow is isomorphic to $H^{i+1}(X, \mathbf{Z}_p(1))[p]$. By Corollary \ref{cor-tocriterion}, the cokernel of the right vertical arrow is isomorphic to $F[p]$.
\end{proof}

\section{An analog of the Igusa surface}
\label{section-theexample}

Fix an algebraically closed field $k$ of characteristic $2$.  
Let $A, B$ be ordinary abelian varieties over $k$ of dimensions $d, e \geq 1$ respectively. Let $Y = A \times B$ and write $N = d + e$ for the dimension of $Y$. Let $x \in B(k)$ be a non-zero $2$-torsion point. Let $G = \mathbf{Z}/2$ act on $Y$ via the involution
$$
g = (-1) \times \tau_x : A \times B \to A \times B.
$$
This is a free action of a finite group on a smooth projective variety, so the quotient $\pi: Y \to X = Y/G$ is a smooth projective variety over $k$ which is ordinary by the following. 

\begin{lemma}
\label{lemma-quotientofordinary}
    Let $Y$ be a smooth proper scheme over an algebraically closed field of positive characteristic $p$. Assume $Y$ is ordinary. Let $G$ be a finite group acting freely on $Y$. Then $Y/G$ is ordinary also.
\end{lemma}

Note that in the generality above $Y/G$ is an algebraic space smooth and proper over the base field but there is no essential difficulty extending the notion of ordinarity to this case.

\begin{proof}
Because we work in prime characteristic, the abelian sheaves $B^j$ have natural $\mathcal{O}$-module structures given by $f\cdot  d \omega = f^pd\omega = d(f^p \omega)$. Equivalently, we view $B^j$ as the image of the $\mathcal{O}$-linear map $F_*(d^{j-1}) : F_*\Omega^{j-1} \to F_*\Omega^j$. Write $\pi : Y \to Y/G$ for the quotient map. Then since $\pi$ is \'etale, we have $\pi^*B^j_{Y/G} = B^j_Y$ as $\mathcal{O}_Y$-modules for all $j$: The square
\[\begin{tikzcd}
	Y & Y \\
	{Y/G} & {Y/G}
	\arrow["F", from=1-1, to=1-2]
	\arrow["\pi"', from=1-1, to=2-1]
	\arrow["\pi", from=1-2, to=2-2]
	\arrow["F", from=2-1, to=2-2]
\end{tikzcd}\]
in which $F$ denotes absolute Frobenius is Cartesian since $\pi$ is \'etale. By flat base change and the fact that $\pi^*\Omega^k_{Y/G} = \Omega^k_Y$, we see that $\pi^* F_*\Omega^k_{Y/G} = F_*\Omega^k_Y$. This shows that the natural morphism of complexes of $\mathcal{O}_Y$-modules $\pi^*F_*\Omega^\bullet_{Y/G} \to F_*\Omega_Y^\bullet$ is an isomorphism. Since $\pi^*$ is exact, taking images of each differential proves the claim.

Since $\pi^*B^j_{Y/G} = B^j_Y$, there is a spectral sequence
    $$
    E_2^{r, s} = H^r(G, H^s(Y, B^j_Y)) \implies H^{r+s}(Y/G, B^j_{Y/G}).
    $$
Since $Y$ is ordinary we have $E_2^{r, s} = 0$ for all $r, s$ so the spectral sequence converges to zero and $Y/G$ is ordinary. 
\end{proof}

The following is standard.

\begin{lemma}
    The cotangent sheaf of $X$ is free, $\Omega^1_X \cong \mathcal{O}_X^{\oplus N}$.
\end{lemma}

\begin{proof}
    Since $Y$ is an abelian variety over $k$, its cotangent sheaf is free
    $$
    H^0(Y, \Omega^1_Y) \otimes _k \mathcal{O}_Y = \Omega^1_Y.
    $$
    The group $G$ acts trivially on $H^0(Y, \Omega^1_Y)$ so in fact the isomorphism above is an isomorphism of $G$-equivariant sheaves, proving the lemma.
\end{proof}

\begin{lemma}
\label{lemma-picardschemecomp}
    There is an exact sequence of group schemes
    $$
0 \to B/\langle x \rangle ^t \to \operatorname{Pic}^\tau_X \to A^t[2] \to 0.
$$
    In particular, $\operatorname{dim}_k H^1(X, \mathcal{O}_X) = N$.
\end{lemma}

\begin{proof}
    The second statement follows from the first by computing the Lie algebras of both sides (the fact that $B/\langle x \rangle ^t$ is smooth implies that $\operatorname{Pic}^\tau_X \to A^t[2]$ induces a surjection of Lie algebras). We prove the first. By \cite[Theorem 2.1]{MR512270}, there is an exact sequence of group schemes
    $$
    0 \to \mu_2 \to \operatorname{Pic}_X \to (\operatorname{Pic}_Y)^{G} \to 0
    $$
    and hence also an exact sequence
    $$
    0 \to \mu_2 \to \operatorname{Pic}^{\tau}_X \to (\operatorname{Pic}^{\tau}_Y)^{G} \to 0.
    $$
    We have $\operatorname{Pic}^\tau_Y = A^t\times B^t$ where $G$ acts by $(-1)$ on the first component and trivially on the second. Hence $(\operatorname{Pic}^{\tau}_Y)^{G} = A^t[2] \times B^t$. We may also apply Jensen's result to the quotient $B/\langle x \rangle$ and we obtain a commutative diagram with exact rows
\[\begin{tikzcd}
	0 & {\mu_2} & {B/\langle x \rangle ^t} & {B^t} & 0 \\
	0 & {\mu_2} & {\operatorname{Pic}^\tau _X} & {A^t[2] \times B^t} & 0.
	\arrow[from=1-1, to=1-2]
	\arrow[from=1-2, to=1-3]
	\arrow[from=1-2, to=2-2]
	\arrow[from=1-3, to=1-4]
	\arrow[from=1-3, to=2-3]
	\arrow[from=1-4, to=1-5]
	\arrow[from=1-4, to=2-4]
	\arrow[from=2-1, to=2-2]
	\arrow[from=2-2, to=2-3]
	\arrow[from=2-3, to=2-4]
	\arrow[from=2-4, to=2-5]
\end{tikzcd}\]
The left vertical arrow is the identity and the right vertical arrow is the inclusion of the second summand. Hence the Snake Lemma gives the exact sequence of the statement.
\end{proof}

\begin{lemma}
    We have $H^1(X, \mu_2) \cong (\mathbf{Z}/2)^{\oplus N}, H^2(X, \mu_2) \cong (\mathbf{Z}/2)^{\oplus N^2}$.
\end{lemma}

\begin{proof}
    We have $H^0(X, \Omega^1_X) \cong k^{\oplus N}$ and $H^1(X, \Omega^1_X) \cong k^{\oplus N^2}$ by the preceding lemmas, so the result follows from Lemma \ref{lemma-computation}.
\end{proof}

\begin{lemma}
    We have $H^1(X, \mathbf{Q}_2(1)) \cong \mathbf{Q}_2^{\oplus e}, H^2(X, \mathbf{Q}_2(1)) \cong \mathbf{Q}_2^{\oplus (d^2+e^2)}$.
\end{lemma}

\begin{proof}
    Write $W = W(k)$ and $K = \operatorname{Frac}(W)$. We have $H^1_{crys}(X/W)_K = (H^1_{crys}(Y/W)_K)^G$ and $H^1_{crys}(Y/W)_K = H^1_{crys}(A/W)_K \oplus H^1_{crys}(B/W)_K$ where $G$ acts on the first component by multiplication by $-1$ and on the second trivially. Hence $H^1_{crys}(X/W)_K = H^1(B/W)_K$. By \cite[II, Th\'eor\`eme 5.5]{MR565469}, the dimension over $\mathbf{Q}_2$ of $H^1(X, \mathbf{Q}_2(1))$ is equal to the dimension over $K$ of the slope $1$ part of $H^1(X/W)_K = H^1(B/W)_K$ which is $e$ since $B$ is an ordinary abelian variety of dimension $e$.

    Next, 
    $$H^2_{crys}(Y/W)_K = H^2_{crys}(A/W)_K \oplus  (H^1(A/W)_K \otimes   H^1(B/W)_K) \oplus H^2_{crys}(B/W)_K$$
    and $H^2_{crys}(A/W)_K = \bigwedge^2H^1_{crys}(A/W)_K$ and similarly for $B$. Computing $G$-invariants as above we see
    $$
    H^2_{crys}(X/W)_K = H^2_{crys}(A/W)_K \oplus  H^2_{crys}(B/W)_K.
    $$
    Since $A, B$ are ordinary abelian varieties, the slope one parts of the two components have dimensions $d^2, e^2$ respectively. Applying \cite[II, Th\'eor\`eme 5.5]{MR565469} again gives the result.
\end{proof}

\begin{corollary}
\label{corollary-brauermod2}
    We have $H^2(X, \mathbf{G}_m)/2H^2(X, \mathbf{G}_m) \cong (\mathbf{Z}/2)^{\oplus M}$ where 
    $$
    M = 2de - d.
    $$
\end{corollary}

\begin{proof}
    From universal coefficients (Lemma \ref{lemma-univcoeff}), Lemma \ref{lemma-finitetypeifordinary} and the previous two Lemmas we deduce:
    $$
    H^1(X, \mathbf{Z}_2(1)) \cong \mathbf{Z}_2^{\oplus e}, H^2(X, \mathbf{Z}_2(1))[2] \cong (\mathbf{Z}/2)^{\oplus d}, H^2(X, \mathbf{Z}_2(1))/2 \cong (\mathbf{Z}/2)^{\oplus (d^2+e^2+d)},
    $$
    and finally $H^3(X, \mathbf{Z}_2(1))[2] \cong (\mathbf{Z}/2)^{\oplus M}$ where 
    $$
    M = N^2 - (d^2+e^2+d) = 2de-d.
    $$
   The group $H^2(X, \mathbf{G}_m)$ is a torsion abelian group by \cite[Proposition 1.4]{MR1608805} and its $2$-primary part has 
    $$
    H^2(X, \mathbf{G}_m)[2^{\infty}] \cong (\mathbf{Q}_2/\mathbf{Z}_2)^{\oplus e} \oplus F
    $$
    where $e \geq 0$ is an integer and $F$ is a finite abelian $2$-group by Theorem \ref{lemma-cofinite}. Hence
    $$
    H^2(X, \mathbf{G}_m)/2H^2(X, \mathbf{G}_m) \cong F/2F.
    $$
    Since $F$ is a finite abelian $2$-group there is a non-canonical isomorphism $F/2F \cong F[2]$ and so
    $$
    H^2(X, \mathbf{G}_m)/2H^2(X, \mathbf{G}_m) \cong H^3(X, \mathbf{Z}_2(1))[2] \cong (\mathbf{Z}/2)^{\oplus M}
    $$
    by Lemma \ref{lemma-recoveref}.
\end{proof}

We also compute the following.

\begin{lemma}
\label{lemma-cohomologyofpic}
    Write $T = \operatorname{Spec}(k[\epsilon]).$ Then 
    $$
    H^1(T, \operatorname{Pic}_{X}) \cong H^1(A, \mathcal{O}_A)
    $$
    as vector spaces over $k$.
\end{lemma}

\begin{proof}
    We use repeatedly that for a smooth group scheme $G$ we have $H^i(T, G) = 0$ for $i > 0$ by Grothendieck's Theorem \cite[Th\'eor\`eme 11.7]{MR244271}. Since $k = \bar{k}$, we have
    $$
    \operatorname{Pic}_X \cong \operatorname{Pic}^\tau_X \times \operatorname{NS}(X)/torsion
    $$
    where the second is the constant group scheme associated to a free abelian group of finite rank. The second factor has no cohomology so 
    $$
    H^1(T, \operatorname{Pic}_X) = H^1(T, \operatorname{Pic}^\tau_X).
    $$
    Next, from the exact sequence of Lemma \ref{lemma-picardschemecomp}
    and the fact that $B/\langle x \rangle ^t$ is smooth we have 
    $$
    H^1(T, \operatorname{Pic}_X) = H^1(T, A^t[2]). 
    $$
    The long exact sequence of cohomology associated to the short exact sequence of group schemes
    $$
    0 \to A^t[2] \to A^t \xrightarrow{2} A^t \to 0
    $$
    identifies $H^1(T, A^t[2])$ with the cokernel of multiplication by $2$ on the group $A^t(k[\epsilon]).$ 
    
    To compute this, consider the commutative diagram with exact rows
\[\begin{tikzcd}
	0 & {\operatorname{Lie}(A^t)} & {A^t(k[\epsilon])} & {A^t(k)} & 0 \\
	0 & {\operatorname{Lie}(A^t)} & {A^t(k[\epsilon])} & {A^t(k)} & 0
	\arrow[from=1-1, to=1-2]
	\arrow[from=1-2, to=1-3]
	\arrow["2"', from=1-2, to=2-2]
	\arrow[from=1-3, to=1-4]
	\arrow["2"', from=1-3, to=2-3]
	\arrow[from=1-4, to=1-5]
	\arrow["2"', from=1-4, to=2-4]
	\arrow[from=2-1, to=2-2]
	\arrow[from=2-2, to=2-3]
	\arrow[from=2-3, to=2-4]
	\arrow[from=2-4, to=2-5].
\end{tikzcd}\]
The left vertical map is zero as $k$ has characteristic $2$. Since multiplication by $2$ is a group homomorphism, we see that it induces the zero map on all tangent spaces. The map on kernels $A^t(k[\epsilon])[2] \to A^t(k)[2]$ is surjective because if $x: \operatorname{Spec}(k) \to A^t[2]$ then we may consider the lift of $x$ to $\operatorname{Spec}k[\epsilon]$ corresponding to the zero tangent vector, and this is a $2$-torsion element of $A^t(k[\epsilon])$ by the previous sentence. Also, since $k$ is algebraically closed and multiplication by $2$ is an isogeny, we have $A^t(k)/2A^t(k) = 0$. Therefore, the snake lemma gives an isomorphism
$$
\operatorname{Lie}(A^t) \cong H^1(T, A^t[2]),
$$
 as needed. 
 \end{proof}

\section{Proof that the example works}
\label{section-proofmaintheorem}

Let $X$ be the variety constructed in the previous section. Assume that we chose the integers $d, e \geq 1$ in such a way that 
$$2de-d > d(e+d).$$
For fixed $d$ this holds whenever $e \gg 0$, and it holds for example when $d = 1, e = 3$.

\begin{thm}
\label{thm-mainpaperversion}
There is an element $\alpha \in H^2(X, \mathbf{G}_m)$ for which the map $a$ of Lemma \ref{lemma-exactsequence} has non-zero differential at the identity element. In particular, the map $a$ is not zero on $k[\epsilon]$-points, and for any $\mathbf{G}_m$-gerbe $\mathcal{X} \to X$ with cohomology class $\alpha$ the homomorphism of group schemes $\operatorname{Aut}^0_{\mathcal{X}} \to \operatorname{Aut}^0_X$ from (\ref{equn-olssonsequence}) is not surjective. 
\end{thm}

\begin{remark}
\label{remark-classicaligusasurface}
       If instead we took $d = e = 1$, then $H^2(X, \mathcal{O}_X) \cong k \cong H^1(k[\epsilon], \operatorname{Pic}_X)$ by the fact that $X$ has trivial tangent bundle and Lemma \ref{lemma-cohomologyofpic}. Thus condition (2) of Theorem \ref{thm-sufficientconditions} holds and $\operatorname{Aut}^0_{\mathcal{X}} \to \operatorname{Aut}^0_X$ is surjective for any $\mathbf{G}_m$-gerbe $\mathcal{X} \to X$.
\end{remark}

\begin{proof}
    The second sentence follows from the first: If the differential of $a$ is not zero then by definition $a$ is non-zero on $k[\epsilon]$-points and $a|_{\operatorname{Aut}^0_X}$ is not zero as a map of sheaves. From Lemma \ref{lemma-exactsequence}, the fact that $a|_{\operatorname{Aut}^0_X}$ is not zero implies that $\operatorname{Aut}^0_{\mathcal{X}} \to \operatorname{Aut}^0_X$ is not surjective. Let us prove the first statement. 

    \underline{Claim 1:} It suffices to show that the image of 
    \begin{equation}
    \label{equn-dlogmap}
    \operatorname{dlog} : H^2(X, \mathbf{G}_m) \to H^2(X, \Omega^1_X) \cong H^2(X, \mathcal{O}_X) \otimes _k H^0(X, \Omega^1_X)
    \end{equation}
    is not contained in the subspace $H^1(\operatorname{Spec}(k[\epsilon], \operatorname{Pic}_{X}) \otimes _k H^0(X, \Omega^1_X)$ where the inclusion $H^1(\operatorname{Spec}(k[\epsilon], \operatorname{Pic}_{X}) \hookrightarrow H^2(X, \mathcal{O}_X)$ is from Theorem \ref{thm-raynaud}. 

    If there is $\alpha \in H^2(X, \mathbf{G}_m)$ such that $\operatorname{dlog} (\alpha) = \sum_{i = 1}^k a_i \otimes b_i$ and $a_1 \notin H^1(\operatorname{Spec}(k[\epsilon], \operatorname{Pic}_{X})$ and $b_1, \dots , b_k \in H^0(X, \Omega^1_X)$ are linearly independent, then since $\Omega^1_X$ is free, we can find $\theta \in H^0(X, T_X) = H^0(X, \Omega^1_X)^*$ such that $\theta(b_1) \neq 0$ but $\theta(b_i) = 0$ for $i > 1$ and then $\theta(\operatorname{dlog}(\alpha)) \notin H^1(\operatorname{Spec}(k[\epsilon], \operatorname{Pic}_{X}).$ Then by Lemma \ref{lemma-derivative}, for any gerbe $\mathcal{X} \to X$ with class $\alpha$ the differential of $a$ at the identity is not zero. 

    \underline{Claim 2:} The $k$-span of the image of (\ref{equn-dlogmap}) has dimension $2de-d$.

    The map (\ref{equn-dlogmap}) factors as
    $$
    H^2(X, \mathbf{G}_m) \to H^2(X, \mathbf{G}_m)/2 \xhookrightarrow{i} H^3(X, \mu_2) \xhookrightarrow{\varphi} H^2(X, \Omega^1_X)
    $$
    where $i$ comes from the Kummer sequence and $\varphi$ is the inclusion from Remark \ref{remark-extensionofscalars}. That remark showed that the image $I$ of $\varphi$ satisfies $I \otimes _{\mathbf{F}_2} k = H^2(X, \Omega^1_X)$. The group $H^2(X, \mathbf{G}_m)/2$ is an $\mathbf{F}_2$-subspace of $H^3(X, \mu_2)$ of dimension $2de-d$ by Corollary \ref{corollary-brauermod2}, so it follows that the image of $H^2(X, \mathbf{G}_m)/2$ in $H^2(X, \Omega^1_X)$ spans a $k$-subspace of dimension $2de-d$ as claimed.

    Now the image of (\ref{equn-dlogmap}) cannot be contained in the subspace $H^1(\operatorname{Spec}(k[\epsilon]), \operatorname{Pic}_{X}) \otimes _k H^0(X, \Omega^1_X)$ because this has dimension $d(e+d)$ by Lemma \ref{lemma-cohomologyofpic} and this concludes the proof.
\end{proof}

\bibliographystyle{alpha}
\bibliography{references}

\textsc{UC Berkeley Department of Mathematics, Department of Mathematics
970 Evans Hall, MC 3840
Berkeley, CA 94720-3840}

Email: \href{nolander@berkeley.edu}{nolander@berkeley.edu}

\end{document}